\documentclass[reqno,  a4paper, 11pt,oneside]{amsart}
\usepackage[colorlinks=true]{hyperref}
\usepackage[english]{babel}
\usepackage{amsmath}
\usepackage{amsfonts}
\usepackage{thmtools}
\usepackage{amssymb}
\usepackage{tikz-cd}
\usepackage{amsthm}
\usepackage{stmaryrd}
\usepackage{microtype}
\usepackage{mathrsfs}
\usepackage{enumitem}

\usepackage[nameinlink]{cleveref}

\newcount\hh
\newcount\mm
\mm=\time
\hh=\time
\divide\hh by 60
\divide\mm by 60
\multiply\mm by 60
\mm=-\mm
\advance\mm by \time
\def\hhmm{\number\hh:\ifnum\mm<10{}0\fi\number\mm}

\DeclareMathOperator{\pr}{pr}

\DeclareMathOperator{\rs}{rs}

\DeclareMathOperator{\sep}{sep}

\DeclareMathOperator{\dlog}{dlog}

\DeclareMathOperator{\cont}{cont}

\DeclareMathOperator{\coker}{coker}
\DeclareMathOperator{\Spec}{Spec}

\DeclareMathOperator{\Char}{char}
\DeclareMathOperator{\Aut}{Aut}

\DeclareMathOperator{\Hom}{Hom}
\DeclareMathOperator{\HHom}{\mathscr{H}\!\!\text{\emph{om}}}

\DeclareMathOperator{\Strat}{Strat}

\DeclareMathOperator{\Gal}{Gal}

\DeclareMathOperator{\id}{id}

\DeclareMathOperator{\gp}{gp}

\DeclareMathOperator{\red}{red}

\DeclareMathOperator{\Frac}{Frac}

\newcommand{\Z}{\mathbb{Z}}

\newcommand{\N}{\mathbb{N}}

\newcommand{\llparen}{(\!(}
\newcommand{\rrparen}{)\!)}

\renewcommand{\subset}{\subseteq}

\newcommand{\aut}{\underline{\Aut}^\otimes}
\renewcommand{\phi}{\varphi}

\declaretheoremstyle[spaceabove=15pt,spacebelow=15pt, qed=$\square$]{defstyle}
\declaretheoremstyle[spaceabove=15pt, spacebelow=15pt, bodyfont=\itshape, qed=$\square$]{thmstyle}
\declaretheorem[style=thmstyle,numberwithin=section]{Theorem}
\declaretheorem[style=thmstyle, 
name=Main Theorem,
numbered=no
]{MainTheorem}
\declaretheorem[sibling=Theorem, style=thmstyle]{Proposition}
\declaretheorem[sibling=Theorem,  style=thmstyle]{Lemma}
\declaretheorem[sibling=Theorem, style=thmstyle]{Corollary}

\declaretheorem[sibling=Theorem, style=defstyle]{Definition}

\declaretheorem[sibling=Theorem, style=defstyle]{Remark}

\declaretheorem[style=thmstyle,numbered=no]{Claim}
\declaretheorem[sibling=Theorem, style=defstyle]{Example}
\declaretheorem[numbered=no]{Acknowledgements}

\setenumerate{label={(\alph*)}, ref=(\alph*)}

\title{The homotopy sequence for regular singular stratified bundles}
\author{Giulia Battiston}
\address[Giulia Battiston]{ Ruprecht-Karls-Universität Heidelberg, Mathematisches Institut, Im Neuenheimer Feld 288, D-69120 Heidelberg }
\email{gbattiston@mathi.uni-heidelberg.de}

\author{Lars Kindler}
\address[Lars Kindler]{Freie Universit\"at Berlin,
	Fachbereich Mathematik und Informatik,
	Arnimallee 3,
	10495 Berlin, Germany
}
\email{kindler@math.fu-berlin.de}

\thanks{The first author is supported by the MAthematical Center Heidelberg (MATCH).\\The second author gratefully acknowledges support of the DFG (``Deutsche Forschungsgemeinschaft'') and Harvard University.}
\date{\today{ }\hhmm}

\begin{document}
\begin{abstract}
	A separable, proper morphism of varieties with geometrically connected fibers induces a homotopy exact sequence relating the \'etale fundamental groups of source, target and fiber. Extending work of dos Santos, we prove the existence of an analogous homotopy exact sequence for fundamental group schemes classifying regular singular stratified bundles, under the additional assumption that the morphism in question can be (partially) compactified to a log smooth morphism.
\end{abstract}

\maketitle
\section{Introduction}
As proved in  \cite{SGA1}, if $f:Y\to X$ is a separable, proper morphism of schemes  with geometrically connected fibers, and if $\bar{y}$  is a geometric point of $Y$, then the \'etale fundamental groups of $Y$, $X$ and of the fiber $Y_{f(\bar{y})}$ are related via the exact sequence
\begin{equation}\label{eq:introSES}\pi_1(Y_{f(\bar{y})},\bar{y})\to\pi_1(Y,\bar{y})\to\pi_1(X,f(\bar{y}))\to 1.\end{equation}
In this article, we establish the existence of an exact sequence analogous to \eqref{eq:introSES} for the affine group schemes classifying regular singular stratified bundles (\Cref{sec:Rs-strat}) on $X,Y$ and $Y_{f(\bar{y})}$.

Sequences similar to \eqref{eq:introSES} have been studied for many different kinds of fundamental groups (see for example \cite{DiProShisho}, \cite{DosSantos/HomotopySES}, \cite{Hoshi}, \cite{Hai}, \cite{Lei/Nori}, \cite{Lei/AlgebraicFG}). This article was inspired by the following two particular examples.
\begin{itemize}
	\item In \cite{Hoshi}, it is proved that if $f:Y\rightarrow X$ is a log smooth morphism of fs log schemes with $X$ connected and log regular, then there is an exact sequence  \eqref{eq:introSES} with the  log fundamental group (which can be seen as a generalization of the tame fundamental group, see \cite[Ex.~4.7, (c)]{Illusie/Overview}) instead of the \'etale fundamental group.
	\item In \cite{DosSantos/HomotopySES}, it is proved that if $f$ is smooth, projective and if $Y$, $X$ are smooth connected varieties over an algebraically closed field $k$ of arbitrary characteristic, then there is a homotopy exact sequence \eqref{eq:introSES} for the affine $k$-group schemes which are obtained via Tannaka duality  from the category of stratified bundles, that is from the category of $\mathcal{O}$-coherent $\mathscr{D}$-modules.
\end{itemize}

The objects studied in this article morally lie in the intersection of the above examples, as the notion of regular singularity for an $\mathcal{O}$-coherent $\mathcal{D}$-module naturally specializes to the notion of a tamely ramified covering (see \Cref{sec:Rs-strat} for details and definitions).

The main result is as follows. If $k$ is an algebraically closed field of positive characteristic and if $X$ is a smooth $k$-variety, let us denote by $\Pi^{\rs}(X,x)$ the affine $k$-group scheme corresponding to the category of regular singular stratified bundles on $X$ with respect to the base point $x\in X(k)$ (see \Cref{defn:stratifiedBundles}). 

\begin{MainTheorem}[{see \Cref{thm:sesFunG}}]\label{thm:main} Let $f:Y\to X$ be a smooth, projective morphism of smooth, connected $k$-varieties with geometrically connected fibers, and assume that $X$ admits a good compactification $\overline{X}$ (\Cref{defn:stratifiedBundles}, \ref{defn:gpc}). Assume furthermore that:
	\begin{enumerate}[label={\emph{(\roman*)}}, ref={(\roman*)}]
		\item There is a good partial compactification  $Y\subset \overline{Y}$ such that $f$ extends to $\bar{f}:\overline{Y}\to\overline{X}$ and $\bar{f}(\overline{Y})$ contains every codimension $1$ point of $\overline{X}$;
		\item\label{item:mainThm:ii} $\bar{f}$ is log smooth with respect to the natural fs log structure induced on $\overline{X}$ and $\overline{Y}$ by their divisors at infinity (see \Cref{rem:log-smoothness}).
	\end{enumerate}

Then, for every $x\in X(k)$ and $y\in Y_x(k)$, there is an exact sequence of $k$-group schemes
\[\Pi(Y_x,y)\xrightarrow{j^*}\Pi^{\rs}(Y,y)\xrightarrow{f^*} \Pi^{\rs}(X,x)\to 1.\]
\end{MainTheorem}
Moreover, the theorem admits a refinement dealing with the notion of regular singularity with respect some specific good partial compactification of $X$, therefore we obtain information which is not conditional on the existence of a good compactification of $X$.
We would like to stress that the log smoothness assumption in \ref{item:mainThm:ii} cannot be dropped entirely. Namely (see \Cref{rmk:counterex}), we prove that the example described in \Cref{sec:ExRa}, which is due to Raynaud, provides a counterexample to the exactness of the homotopy sequence in general.
\

The article is organized as follows. In \Cref{sec:Rs-strat} we recall the definitions of stratified bundles and of regular singularity; in \Cref{sec:log-diffOps} we recall the notion of log smoothness and study the pullback of logarithmic differential operators along such morphisms. In \Cref{sec:ExRa} we present the example of Raynaud.  \Cref{sec:ExCri} establishes a criterion for a sequence of affine $k$-group schemes to be exact, and the proof of the main theorem is carried out in \Cref{sec:Homty}.

\begin{Acknowledgements}We would like to thank H\'el\`ene Esnault for communicating to us the example from \Cref{sec:ExRa}, and we are are grateful to Michel Raynaud for allowing us to include it in this article.
\end{Acknowledgements}

\section{Regular singular stratified bundles}\label{sec:Rs-strat}
Let $k$ be an algebraically closed field of characteristic $p\geq 0$. 
We begin by recalling some basic facts about stratified bundles and regular singularity. For full details we refer to \cite{Gieseker/FlatBundles} and \cite{Kindler/FiniteBundles}. 
\begin{Definition}\label{defn:stratifiedBundles}
	Let $X$ be a smooth, separated, finite type $k$-scheme. 
	\begin{enumerate}
		\item A \emph{stratified bundle on $X$} is a left-$\mathscr{D}_{X/k}$-module $E$ which is coherent as an $\mathscr{D}_{X/k}^{\leq 0}=\mathcal{O}_X$-module. Here $\mathscr{D}_{X/k}$ is the sheaf of differential operators defined in \cite[\S16]{EGA4}, and $\mathscr{D}_{X/k}^{\leq n}$, $n\geq 0$, the subsheaf of operators of order at most $n$. The category of stratified bundles with morphisms the morphisms of left-$\mathscr{D}_{X/k}$-modules, is denoted by $\Strat(X)$.
		\item\label{defn:gpc} If $\overline{X}$ is a smooth, separated, finite type $k$-scheme, such that there is an open immersion $X\hookrightarrow \overline{X}$, such that $\overline{X}\setminus X$ is  the support of a strict normal crossings divisor, then the pair $(X,\overline{X})$ is called a \emph{good partial compactification of $X$}. If $\overline{X}$ is proper, then $(X,\overline{X})$ is called a \emph{good compactification of $X$}.
		\item\label{item:defn:regSing}  If $(X,\overline{X})$ is a good
			partial compactification of $X$ and if 
			$D:=(\overline{X}\setminus X)_{\red}$ is the associated strict normal crossings divisor, then an object
			$E\in \Strat(X)$ is said to be
			\emph{$(X,\overline{X})$-regular singular} if there
			exists a $\mathscr{D}_{\overline{X}/k}(\log D)$-module
			$E$, which is torsion-free and coherent with respect to
			the induced $\mathcal{O}_{\overline{X}}$-module structure,
			such that there is an isomorphism
			\[E\cong \overline{E}|_X\]
			in $\Strat(X)$. The sheaf $\mathscr{D}_{\overline{X}/k}(\log D)$ is the sheaf of subrings of $\mathscr{D}_{\overline{X}/k}$ generated by operators fixing all powers of the ideal of $D$. For more details see \cite[p.~17]{Gieseker/FlatBundles} or \cite[\S3]{Kindler/FiniteBundles}.

			We denote by $\Strat^{\rs}( (X,\overline{X}))$ the full
			subcategory of $\Strat(X)$ with objects the
			$(X,\overline{X})$-regular singular stratified bundles.
			If $X$ is connected, then after fixing a base point $x\in X(k)$, $\Strat^{\rs}((X,\overline{X}))$ is a neutral Tannakian category over $k$ (\cite[Prop.~4.5]{Kindler/FiniteBundles}) and we denote the associated group scheme by $\Pi_1^{\rs}((X,\overline{X}),x)$. If $(X,\overline{X})$ is a good compactification, then this group scheme is independent of the choice of $\overline{X}$ (\cite[Thm.~3.13]{Gieseker/FlatBundles}), so  we also write $\Pi_1^{\rs}(X,x)$.
	\end{enumerate}
\end{Definition}
It is proved in \cite[Thm.~5.2]{Kindler/FiniteBundles} that $\overline{E}$ in \Cref{defn:stratifiedBundles}, \ref{item:defn:regSing} can be required to be a locally free $\mathcal{O}_X$-module without changing the definition.
In this article, we also use the following simple criterion.

\begin{Lemma}\label{lemma:rs-criterion}
	Let $X$ be a smooth, separated, finite type $k$-scheme, let $(X,\overline{X})$ a good partial
	compactification, let  $D:=(\overline{X}\setminus X)_{\red}$ be the
	associated strict normal crossings
	divisor, and  let  $j:X\hookrightarrow \overline{X}$ be the associated open
	immersion. Let
	$E$ be a stratified bundle on $X$, and let $\overline{E}$ be any locally
	free, coherent $\mathcal{O}_{\overline{X}}$-module extending $E$. Then the adjunction
	map $\overline{E}\rightarrow
	j_*j^*\overline{E}=j_*E$ is injective and we consider $\overline{E}$ as a
	sub-$\mathcal{O}_{\overline{X}}$-module of $j_*E$. Recall that $j_*E$ carries
	a natural 
	$\mathscr{D}_{\overline{X}/k}(\log D)$-action. The following 
	statements are equivalent.
	\begin{enumerate}[label={\emph{(\alph*)}}, ref={(\alph*)}]
		\item\label{item:rs-criterion1} $E$ is $(X,\overline{X})$-regular
			singular, 
		\item\label{item:rs-criterion2} The
			$\mathscr{D}_{\overline{X}/k}(\log D)$-submodule of $j_*E$
			generated by $\overline{E}$ is $\mathcal{O}_{\overline{X}}$-coherent.
	\end{enumerate}
\end{Lemma}
\begin{proof}
	Only \ref{item:rs-criterion1} $\Rightarrow$ \ref{item:rs-criterion2} is
	nontrivial. We may assume that $\overline{X}=\Spec A$ is affine and $D$ regular, defined by $t\in A$. Assume that $E$ is $(X,\overline{X})$-regular singular. According \cite[Thm.~5.2]{Kindler/FiniteBundles} there exists a locally free, coherent $\mathcal{O}_{\overline{X}}$-module $\overline{E}'$ with an action of $\mathscr{D}_{\overline{X}/k}(\log D)$, such that $\overline{E}'|_X\cong E$ as $\mathscr{D}_{X/k}$-modules.
	Then there exists some $N\geq 0$, such that $\overline{E}\subset
	t^{-N}\overline{E}'$ as coherent $\mathcal{O}_{\overline{X}}$-submodules of $j_*E$.  Note that
	$t^{-N}\overline{E}'$ is also a $\mathscr{D}_{\overline{X}/k}(\log D)$-submodule
	of $j_*E$. Thus the $\mathscr{D}_{X/k}(\log D)$-submodule of $j_*E$
	generated by $\overline{E}$ is contained in $t^{-N}\overline{E}'$ and
	hence coherent, which is what we wanted to prove.
\end{proof}
\section{Logarithmic differential operators and logarithmic smoothness}\label{sec:log-diffOps}
For the details of the theory of logarithmic schemes we refer to \cite{Kato/LogSchemes}, but we briefly recall the facts that we use in this article.
More details can also be found in \cite[2.1]{Kindler/Thesis}.
\begin{Remark}\label{rem:log-smoothness}
	The notion of an fs (fine and saturated) log structure can be understood as a generalization of the notion of a good partial compactification (\Cref{defn:stratifiedBundles}).
	Let $k$ be a field. 
		If $X$ is a smooth, separated, finite type $k$-scheme and
			$(X,\overline{X})$ a good partial compactification of
			$X$, write $M_{\overline{X}}:=M_{(X,\overline{X})}$ for
			the presheaf given by
\[M_{(X,\overline{X})}(U):=\{g\in \mathcal{O}_{\overline{X}}(U)\mid g|_{U\cap X}\in \mathcal{O}_{X}(X\cap U)^\times\},\]
for $U\subset \overline{X}$ open. This is in fact a sheaf of monoids and the natural map $\mathcal{O}_{\overline{X}}^\times \rightarrow M_{\overline{X}}$
 makes $(\overline{X},M_{\overline{X}})$ into an fs log scheme.

It is easy to write down local charts for this log scheme. If ${U}=\Spec
A\subset \overline{X}$ is an affine open subset such that the boundary divisor
is defined by $V(t_1\cdots\ldots\cdot t_r)$, where $t_1,\ldots, t_r$ are part
of a regular system of parameters for $A$, then we obtain a morphism of monoids
$\N^r\rightarrow M_{\overline{X}}(U)$, $(m_1,\ldots, m_r)\mapsto \prod_{i=1}^rt_i^{m_i}$. This
defines a chart ${(U, M_{\overline{X}}|_U)}\rightarrow \Spec \Z[\N^r]$.

More generally, given a finite set of units $u_1,\ldots, u_n\in A^\times$, the morphism of monoids
\[\Z^n\oplus \N^r\rightarrow M_{\overline{X}}(U), (e_1,\ldots, e_n, m_1,\ldots, m_r)\mapsto \prod_{i=1}^nu_i^{e_i}\prod_{i=1}^r t^{m_i}_i\]
also induces a chart.

For a particular example of a chart for a morphism arising from good compactifications, and for what it means for such a morphism to be log smooth, see \Cref{ex:logSmooth}.
\end{Remark}

\begin{Definition}\label{defn:logPP}
As in the non-log case, if $(X,M_X)$ is an fs $(S,M_S)$-log scheme, for every $n\geq 0$, one defines \emph{the sheaf $\mathscr{P}^n_{(X,M_X)/(S,M_S)}$ of logarithmic principal parts of order $n$} as the structure sheaf of the $n$-th order thickening of a certain log diagonal (\cite[Rem.~5.8]{Kato/LogSchemes}). It is equipped with three ring homomorphisms 
\begin{equation*}
	\begin{tikzcd}
		\mathcal{O}_X\rar{d_i}& \mathscr{P}_{(X,M_X)/(S,M_S)}^n\rar{\Delta}&\mathcal{O}_X
	\end{tikzcd}
\end{equation*}
where $i=1,2$ and such that $\id=\Delta\circ d_1=\Delta\circ d_2$.

The \emph{the sheaf of logarithmic differential operators of order $\leq n$} is then defined as  
\[\mathscr{D}^{\leq n}_{(X,M_X)/(S,M_S)}=\HHom_{\mathcal{O}_X}(\mathscr{P}^{n}_{(X,M_X)/(S,M_S)},\mathcal{O}_X),
\] and the \emph{sheaf of logarithmic differential operators} as \[\mathscr{D}_{(X,M_X)/(S,M_S)}=\bigcup_n \mathscr{D}_{(X,M_X)/(S,M_S)}^{\leq n}.\]
\end{Definition}

\begin{Proposition}\label{prop:localFreeness}
	If $f:(X,M_X)\rightarrow (S,M_S)$ is a log smooth morphism of fs log
	schemes, then for every $n\geq 0$, the sheaf of logarithmic principal
	parts $\mathscr{P}^n_{(X,M_X)/(S,M_S)}$ is locally free of finite rank
	with respect to both its left and right $\mathcal{O}_X$-structures.

	More precisely, let $d_1,d_2:\mathcal{O}_X\rightarrow \mathscr{P}_{(X,M_X)/(S,M_S)}^n$ denote the two structure maps. 
	Let $\bar{x}$ be a geometric point of $X$ and
	let $m_1,\ldots, m_r\in M_{X,\bar{x}}$ be elements such that
	$\dlog(m_1),\ldots, \dlog(m_r)$ freely generate
	$\Omega^1_{(X,M_X)/(S,M_S),\bar{x}}$. Then $d_1(m_i)=d_2(m_i)\cdot u_i$ with $u_i\in (\mathscr{P}_{(X,M_X)/(S,M_S)}^n)^\times_{\bar{x}}$.
	The assignment $m_i\mapsto u_i$ extends to a functorial morphism of monoids $\mu:M_X\rightarrow (\mathscr{P}^n_{(X,M_X)/(S,M_S)},\cdot)$
	and $\mathscr{P}^n_{(X,M_X)/(S,M_S),\bar{x}}$ is freely generated as
	either left- or right-$\mathcal{O}_{X,\bar{x}}$-module by monomials of
	degree $\leq n$ in $1-\mu(m_1),\ldots, 1-\mu(m_r)$. 
\end{Proposition}
\begin{proof}
	The proof is completely analogous to the proof of \cite[Prop.~6.5]{Kato/LogSchemes}, forgetting about divided powers.
\end{proof}

\begin{Proposition}\label{prop:surjective}
	Let $(S,M_S)$ be a fs log scheme and
	let $f:(Y,M_Y)\rightarrow (X,M_X)$ be a log smooth $(S,M_S)$-morphism
	of fs log schemes which are log smooth over $(S,M_S)$. 	Then the natural map
	\[\mathscr{D}_{(Y,M_Y)/(S,M_S)}\rightarrow f^*\mathscr{D}_{(X,M_X)/(S,M_S)}\]
	is surjective.
\end{Proposition}
\begin{proof}
	The argument is almost the same as in the non-log case and essentially in \cite{Montagnon}. We recall it for completeness.

	By definition we have
	\[\mathscr{D}^{\leq n}_{(X,M_X)/(S,M_S)}=\HHom_{\mathcal{O}_X}(\mathscr{P}^n_{(X,M_X)/(S,M_S)},\mathcal{O}_X),\]
	where $n\geq 0$ and $\mathscr{P}^n_{(X,M_X)/(S,M)}$ is the sheaf of
	the $n$-th log principal parts.
	As $(X,M_X)$ is log smooth over $(S,M_S)$, the $\mathcal{O}_X$-modules
	$\mathscr{P}^n_{(X,M_X)/(S,M_S)}$ are locally free with respect to both
	their left- and right-$\mathcal{O}_X$-structure. It follows that it is
	enough to show that for every $n\geq 0$ the canonical morphism
	\begin{equation}\label{eq:principal-parts-map}
		f^*\mathscr{P}_{(X,M_X)/(S,M_S)}^n\rightarrow
		\mathscr{P}^n_{(Y,M_Y)/(S,M_S)}
	\end{equation}
	is injective. This follows from the fact (\cite[3.12]{Kato/LogSchemes})
	that the natural morphism
	\begin{equation}
		\label{eq:split-map}f^*\Omega^1_{(X,M_X)/(S,M_S)}\rightarrow \Omega_{(Y,M_Y)/(S,M_S)}^1
	\end{equation}
	is injective and that its image is locally a direct summand.

	Indeed, \'etale locally on $X$, there exist sections $m_1,\ldots, m_r$ of $M_X$,
	such that $\dlog(m_1)\ldots, \dlog(m_r)$ freely generate
	$\Omega_{(X,M_X)/(S,M_S)}^1$. As \eqref{eq:split-map} is locally split, this means that locally there are sections
	$m'_{r+1},\ldots, m'_{d}$ of $M_Y$, such that
	$\Omega_{(Y,M_Y)/(S,M_S)}^1$ is freely generated by \[\dlog(f^*m_1)
	,\ldots, \dlog (f^*m_r), \dlog m'_{r+1},\ldots, \dlog m'_d.\]

	Finally, by \Cref{prop:localFreeness} there are natural maps
	\[\mu_{X}:M_X\rightarrow\mathscr{P}^n_{(X,M_X)/(S,M_S)}
	,\quad\mu_{Y}:M_Y\rightarrow \mathscr{P}^n_{(Y,M_X)/(S,M_S)},\]
	such that 
	\begin{enumerate}[label={(\roman*)}]
		\item $\mathscr{P}^n_{(X,M_X)/(S,M_S)}$ is freely generated by
			the set of monomials of degree $\leq n$ in
			\[1-\mu_{X}(m_1),\ldots, 1-\mu_{X}(m_r),\]
		\item $\mathscr{P}^n_{(Y,M_Y)/(S,M_S)}$ is freely generated by
			the set of monomials of degree $\leq n$ in
			\[1-\mu_{Y}(f^*m_1),\ldots,
				1-\mu_{Y}(f^*m_r),1-\mu_{Y}(m'_{r+1}),\ldots,
				1-\mu_{Y}(m'_d).\]
		\item \[f^*\mathscr{P}_{(X,M_X)/(S,M_S)}^n\rightarrow
		\mathscr{P}^n_{(Y,M_Y)/(S,M_S)}\]
	is given by $1-\mu_{X}(m_i)\mapsto 1-\mu_{Y}(f^*m_i)$.
	\end{enumerate}
	This shows that \eqref{eq:principal-parts-map} is injective and that its image locally is a
direct summand of $\mathscr{P}_{(Y,M_Y)/(S,M_S)}^n$.\end{proof}

\begin{Corollary}\label{cor:checkRSafterPullback}
	Let $k$ be an algebraically closed field, let $X,Y$  be smooth, separated, finite type $k$-schemes with good partial compactifications $(X,\overline{X})$, $(Y,\overline{Y})$, and write
	 $D_X:=(\overline{X}\setminus X)_{\red}$ and
	 $D_Y:=(\overline{Y}\setminus Y)_{\red}$. Let $f:Y\rightarrow X$ be a morphism which extends
	to a morphism $\bar{f}:\overline{Y}\rightarrow \overline{X}$ with the following properties:
	\begin{enumerate}[label={\emph{(\alph*)}}, ref={(\alph*)}]
		\item\label{ass:boundary} $\bar{f}(\overline{Y})$ contains all generic points of $D_X$.
		\item $\bar{f}$ is log smooth with respect to the log structures on $\overline{Y}$ and $\overline{X}$ defined by $D_Y$, resp. $D_X$.				
	\end{enumerate}
	If $E$ is a stratified bundle on $X$ such that $f^*E$ is
	$(Y,\overline{Y})$-regular singular, then $E$ is $(X,\overline{X})$-regular singular.
\end{Corollary}
\begin{Remark}
	We keep the notations from \Cref{cor:checkRSafterPullback}.
	One of the main results of \cite{Kindler/FiniteBundles} is that a
	stratified bundle $E$ on $X$ with finite monodromy is $(X,\overline{X})$-regular singular
	if and only if the associated Picard-Vessiot torsor on $X$ (which is just a Galois covering)
	is tamely ramified along $D_X$. Thus, for stratified bundles with finite
	monodromy, \Cref{cor:checkRSafterPullback} is a special case of
	\cite[Prop.~7.7]{Abbes-Saito/Cleanliness}, as first lines of the proof of the theorem show that we can assume that $\bar{f}$ is faithfully flat.
\end{Remark}
\begin{proof}[Proof of \Cref{cor:checkRSafterPullback}]
	Without loss of generality we can assume that $X$ and $Y$ are connected.
	Let $\overline{Y}'\subset \overline{Y}$ be the largest open subset on which
	$\bar{f}$ is flat (\cite[Thm.~11.1.1]{EGA4}). Then $Y\subset \overline{Y}'$, as $f$ is smooth. Moreover, $\bar{f}(\overline{Y}')\subset \overline{X}$ is
	open, $X\subset \bar{f}(\overline{Y}')$,  and according to assumption \ref{ass:boundary},
	$\bar{f}(\overline{Y}')$ contains all generic points of $D_X$. Indeed, if $\eta\in \overline{X}$ is a generic point of $D_X$, then $\mathcal{O}_{\overline{X},\eta}$ is a discrete valuation ring, and thus the dominant morphism $\bar{f}$ is flat in a neighborhood of $\bar{f}^{-1}(\eta)$. If $E$
	is a stratified bundle on $X$, then it is $(X,\overline{X})$-regular
	singular if and only if it is $(X,
	\bar{f}(\overline{Y}'))$-regular singular (\cite[Prop.~4.3]{Kindler/FiniteBundles}). Similarly, if $f^*E$ is $(Y,\overline{Y})$-regular singular, then it is also $(Y,\overline{Y}')$-regular singular. Replacing $\overline{Y}$ by
	$\overline{Y}'$, and	$\overline{X}$ by $\bar{f}(\overline{Y}')$, we may assume that $\bar{f}$ is
	faithfully flat.

	Now the argument is almost identical to the proof of
	\cite[Prop.~6.1]{Kindler/Evidence}, using \Cref{prop:surjective}
	instead of the analogous surjectivity statement for the relative Frobenius morphism.
	We fix notations as in the following diagram:
	\begin{equation*}
		\begin{tikzcd}
			Y\rar[hook]{i}\dar{f}&\overline{Y}\dar{\bar{f}}\\
			X\rar[hook]{j}&\overline{X}
		\end{tikzcd}
	\end{equation*}
	Let $E$ be a stratified bundle on $X$ and assume that $f^*E$ is
	regular singular. Fix a locally free extension $E'\subset j_*E$ of $E$ to
	$\overline{X}$. Denote by $\overline{E}\subset j_*E$ the $\mathscr{D}_{\overline{X}/k}(\log D_X)$-submodule of $j_*E$ generated by $E'$; in other words
	$\overline{E}$ is the
	image of the evaluation map 
	\begin{equation}\label{eq:eval1}
		\mathscr{D}_{\overline{X}/k}(\log
		D_X)\otimes_{\mathcal{O}_{\overline{X}}} E'\rightarrow j_*E.
	\end{equation}
	The $\mathcal{O}_{\overline{X}}$-module $\overline{E}$ is quasi-coherent, and to
	show that $E$ is regular singular, it suffices to show that
	$\overline{E}$ is coherent.  As $\bar{f}$ is faithfully flat by assumption,
	it is enough to show that $\bar{f}^*\overline{E}$ is coherent. Note that
	$\bar{f}^*\overline{E}$ is the image of the pullback of \eqref{eq:eval1} along
	$\bar{f}$.

	Next,  write 
	$G$ for the $\mathscr{D}_{\overline{Y}/k}(\log D_Y)$-submodule of $i_*f^*E=\bar{f}^*j_*E$
	spanned by $\bar{f}^*E'$. In other words, $G$ is the image of the evaluation
	map
	\begin{equation}\label{eq:eval2}\mathscr{D}_{\overline{Y}/k}(\log D_Y)\otimes_{\mathcal{O}_{\overline{Y}}} \bar{f}^*E'\rightarrow
			\bar{f}^*j_*E.\end{equation}
	Note that $(\bar{f}^*E')|_Y\cong f^*E$. Thus, as $f^*E$ is $(X,\overline{X})$-regular singular by
	assumption, \cref{lemma:rs-criterion} shows that $G$ is coherent. We show that $G\cong \bar{f}^*\overline{E}$.

	By the definition of the $\mathscr{D}_{\overline{Y}/k}(\log D_Y)$-action on
	$\bar{f}^*j_*E=i_*f^*E$, the maps \eqref{eq:eval1} and \eqref{eq:eval2} fit into the following
	commutative diagram
	\begin{equation*}
		\begin{tikzcd}
			\bar{f}^*\left(\mathscr{D}_{\overline{X}/k}(\log D_X)\otimes_{\mathcal{O}_{\overline{X}}} E'\right)\rar{\bar{f}^*(\text{\eqref{eq:eval1}})}&\bar{f}^*j_*E\\
			\bar{f}^*\mathscr{D}_{\overline{X}/k}(\log D_X)\otimes_{\bar{f}^{-1}\mathcal{O}_{\overline{X}}} \bar{f}^{-1}E'\uar[-, double equal sign distance ]\\
			\mathscr{D}_{\overline{Y}/k}(\log D_Y)\otimes_{\bar{f}^{-1}\mathcal{O}_{\overline{X}}}\bar{f}^{-1}E'\uar{\gamma\otimes \id}\\
			\mathscr{D}_{\overline{Y}/k}(\log D_Y)\otimes_{\mathcal{O}_{\overline{Y}}}\bar{f}^*E' \uar[-, double equal sign distance ]\rar{\text{\eqref{eq:eval2}}}&i_*f^*E\ar[-, double equal sign distance]{uuu},
		\end{tikzcd}
	\end{equation*}
	where
	\[\gamma:\mathscr{D}_{\overline{Y}/k}(\log D_Y)\rightarrow
		\bar{f}^*\mathscr{D}_{\overline{X}/k}(\log D_X)\]
	is the functoriality morphism for logarithmic differential operators.
	According to \Cref{prop:surjective}, this map is surjective, so the
	images of $\bar{f}^*$\eqref{eq:eval1} and \eqref{eq:eval2} agree, which shows that
	$G=\bar{f}^*\overline{E}$.
\end{proof}

\begin{Example}\label{ex:logSmooth} Here are two fundamental (equicharacteristic) examples for log smoothness. Let
	$k$ be a field of characterstic $p>0$,
	$R=\Spec k\llbracket t \rrbracket$, $S=\Spec R$, $K=k\llparen
	t\rrparen$. 
	\begin{enumerate}
		\item\label{ex:log-smooth-a} Define
			\[X_1:=\Spec R[x_1,\ldots, x_n]/(x_1^{m_1}\cdot\ldots\cdot x_n^{m_n}-t)\]
		with $m_1,\ldots, m_n\geq 0$ and at least one $m_i$ prime to $p$.
		\item\label{ex:log-smooth-b} Define
			\[X_2:=\Spec R[u^{\pm 1}, x_1,\ldots, x_n]/(ux_1^{\ell_1}\cdot\ldots\cdot x_n^{\ell_n}-t)\]
		with $\ell_1,\ldots, \ell_n\in \Z$. Note that if there exists one $\ell_i$
		which is prime to $p$, then we can replace $X_2$ by a Kummer covering
		to arrive in case \ref{ex:log-smooth-a}. Thus we assume that $\ell_1,\ldots, \ell_n\in p\Z$.
\end{enumerate}
Equip $S, X_1$ and $X_2$ with the log structures given by $t=0$. 
	Then $X_1\rightarrow S$ and $X_2\rightarrow S$ are both log smooth. 
	Indeed, $S$ admits the chart $\N\rightarrow R$, $m\mapsto t^m$, and
	$X_1$ admits the chart $\N^n\rightarrow H^0(X_1,\mathcal{O}_{X_1})$,
	$(a_1,\ldots, a_n)\mapsto \prod_{i=1}^n x_i^{a_i}$. The morphism
	$\phi_1:\N\rightarrow \N^n, a\mapsto (am_1,\ldots, am_n)$ induces a
	commutative diagram
	\begin{equation*}
		\begin{tikzcd}
			X_1\rar\dar& \Spec \Z[\N^n]\dar{\phi_1}\\
			S\rar&\Spec \Z[\N],
		\end{tikzcd}
	\end{equation*}
	which is easily checked to be a chart for the morphism $X_1\rightarrow S$. 

	Let $\phi_1^{\gp}:\Z\rightarrow \Z^n$ be the homomorphism of abelian groups associated with $\phi_1$. By \cite[Thm.~3.5]{Kato/LogSchemes}, $X_1$ is log smooth over $S$ if
	and only if $\coker(\phi^{\gp}_1)$ has no $p$-torsion, and if the induced
	morphism $X_1\rightarrow S\times_{\Spec \Z[\N]}\Spec \Z[\N^n]$ is
\'etale. These conditions are easily seen to be satisfied 
 if one of the
$m_1,\ldots, m_n$ is coprime to $p=\Char(k)$.

For $X_2$, we can choose the chart \[\Z\oplus \N^n\rightarrow H^0(X_2,\mathcal{O}_{X_2}), (a_0,a_1,\ldots, a_n)\mapsto u^{a_0}\prod_{i=1}^nx_i^{a_i},\]	and $\phi_2:\N\rightarrow \Z\oplus\N^n$, $\phi_2(a)=(a,a\ell_1,\ldots, a\ell_n)$, induces a chart for the morphism $X_2\rightarrow S$. We see that $X_2\rightarrow S$ is log smooth.

	We conclude the example by verifying \Cref{prop:surjective} in this situation for $X_2=\Spec k\llbracket t\rrbracket[u^{\pm 1},x]/(ux^{\ell}-t)$.
	The (completed) ring of differential operators $\mathscr{D}_{S/k}(\log(t))$ is isomorphic to $\bigoplus_{m\geq 0}k\llbracket t\rrbracket\delta^{(m)}_t$, where $\delta^{(m)}_{t}$ is the differential operator such that $\delta^{(m)}_t(t^r)=\binom{r}{m}t^r$ (i.e., the characteristic free version of $\frac{t^m}{m!}\frac{\partial^m}{\partial t^m}$).
The functions $x$ and $u$ are a system of
	coordinates for $X_2$ relative to $k$, and the (completed) ring of differential operators of $X_2$ relative to $k$ is $\bigoplus_{m,n\geq 0} \left(k\llbracket t\rrbracket\delta_u^{(m)} \oplus k\llbracket t\rrbracket\delta_x^{(n)}\right)$. We compute
	\begin{align*}
		\delta_u^{(m)}(t^r)&=\delta_u^{(m)}(u^rx^{\ell r})\\
		&=\delta_u^{(m)}(u^r)x^{\ell r}\\
		&=\binom{r}{m}u^rx^{\ell r}\\
		&=\binom{r}{m}t^r.
	\end{align*}
	This shows that the map
	\[\bigoplus_{m,n\geq 0} \left(k\llbracket t\rrbracket\delta_u^{(m)} \oplus k\llbracket t\rrbracket\delta_x^{(n)}\right)\rightarrow \bigoplus_{m\geq 0} k\llbracket t\rrbracket\delta_t^{(m)}\]
	induced by $X_2\rightarrow S$ is surjective.
\end{Example}
\section{An Example of Raynaud}\label{sec:ExRa}
The log smoothness condition in \Cref{cor:checkRSafterPullback} cannot be dropped, as the following example shows. It is a variant of \cite[Rem.~9.4.3 (c)]{Raynaud/Picard} and is contained in a letter from M.~Raynaud to H.~Esnault dated May 15, 2009. 

We will show: there is a smooth affine curve $S$ over an algebraically closed field $k$ of characteristic $p>0$, a closed point $s\in S$, a morphism $f:X\rightarrow S$, and a finite map $g:S'\rightarrow S$, such that
\begin{enumerate}
	\item $X$ is regular, hence smooth over $k$,
	\item $f$ is faithfully flat,
	\item If $D=f^{-1}(s)_{\red}$, then $X\setminus D\rightarrow S\setminus \{s\}$ is smooth, projective, with  geometrically connected fibers,
	\item $g:S'\rightarrow S$ is \'etale over $S\setminus \{s\}$,
	\item $g:S'\rightarrow S$ is wildly ramified over $s$,
	\item If $X'$ is the normalization of $X\times_SS'$, then the finite map $g_X:X' \rightarrow X$ is \'etale.
\end{enumerate}
The wildly ramified covering $g$ pulls back to an \'etale covering of $X$, and a fortiori to a tamely ramified covering.
Translated into the language of stratified bundles, this means that if $\Sigma\subset D$ is a $0$-dimensional subset such that
$D\setminus \Sigma$ is regular, then the stratified bundle
$(g_*\mathcal{O}_{S'})|_{S\setminus \{s\}}$, which is irregular over $s$, pulls
back to a stratified bundle on $X\setminus D$ which extends to $X$. In
particular, this pullback is  $(X\setminus \Sigma, D\setminus \Sigma)$-regular singular, so the conclusion of \Cref{cor:checkRSafterPullback} fails.

Let $S$ be an affine smooth curve over an algebraically closed field of characteristic $p>0$ and let
${E}\rightarrow S$ be a family of ordinary elliptic curves. In particular,
${E}$ is an abelian $S$-scheme (the N\'eron model of its generic fiber). After
possibly replacing $S$ by a \'etale open, we may assume that ${E}$ contains an
$S$-subgroup scheme $G$ which is isomorphic to $(\Z/p\Z)_S$. Indeed, according
to \cite[12.3]{KatzMazur}, if $E_{p}$ is the kernel of $E\xrightarrow{\cdot p}E$,
then $E_p$ sits in a short exact sequence $0\rightarrow\ker(F)\rightarrow
E_p\rightarrow \ker(V)\rightarrow 0$, where $F$ and $V$ denote the absolute
Frobenius and Verschiebung. As $E$ is ordinary, $V$ is \'etale and $\ker(V)$ is
a cyclic finite \'etale $S$-group scheme of order $p$. This means that there is
a finite separable extension $L$ of $K(S)$ over which
$(E_p)_L=\ker(F)_L\times_L (\Z/p\Z)_L$. We replace $S$ by the integral closure
of a suitable open subset in $L$, and from now on assume that there is a
$S$-subgroup scheme $G\subset E$ with $G\cong (\Z/p\Z)_S$.

Pick a closed point $s\in S$, write $K=K(S)$, $K_s:=\Frac(\widehat{\mathcal{O}_{S,s}})$, fix separable closures $K^{\sep}\subset (K_s)^{\sep}$. Consider the following commutative diagram of Galois cohomology groups:
\begin{equation*}
	\begin{tikzcd}
		H^1(K,G(K^{\sep}))\dar\ar[dashed]{dr}{\phi}\rar&H^1(K,E(K^{\sep}))\dar\\
		H^1(K_s,G(K_s^{\sep}))\rar&H^1(K_s,E_{K_s}(K_s^{\sep}))
	\end{tikzcd}
\end{equation*}
\begin{Claim}
	There exists an element $\alpha\in H^1(K,G(K^{\sep}))$ of order $p$, such that
	$\phi(\alpha)\neq 0$.
\end{Claim}
\begin{proof}
Write $R=\widehat{\mathcal{O}_{S,s}}$ and let $R^{\sep}$ be the
integral closure of $R$ in $K_s^{\sep}$. If $E_R$ is the elliptic curve
on $\Spec R$ obtained by restricting $E$,  then
\[E_R(K_s^{\sep})=E_R(R^{\sep})\twoheadrightarrow E_{R}(k(s)),\] where
$k(s)=k$ is the algebraically closed residue field of $R$. Moreover, if
$E_R(k(s))$ is equipped with the trivial $\Gal(K_s^{\sep}/K_s)$-action,
we obtain a commutative diagram of continuous
$\Gal(K_s^{\sep}/K_s)$-modules
\begin{equation*}
\begin{tikzcd}
	E_{K_s}(K_s^{\sep})\ar[two heads]{rr} && E(k(s))\\
	G(K^{\sep}_s)\uar[hookrightarrow]\rar[-, double equal sign distance ]&G(K_s)\rar[-, double equal sign distance ]&G(k(s))\uar[hookrightarrow]
	\end{tikzcd}
\end{equation*}
and consequently
\begin{equation*}
	\begin{tikzcd}
		H^1(K_s,E(K_s^{\sep}))\rar&\Hom^{\cont}(\Gal(K^{\sep}_s/K_s),E(k(s)))\\
		H^1(K_s,G(K_s^{\sep}))\uar\rar[-, double equal sign distance ]&\uar[hookrightarrow]\Hom^{\cont}(\Gal(K_s^{\sep}/K_s),G(k(s)))\\
	\end{tikzcd}
\end{equation*}
This shows we can find $\alpha_s\in H^1(K_s, G(K_s^{\sep}))$ of order $p$ with nontrivial
image in $H^1(K_s, E(K_s^{\sep}))$.

Finally, note that the theorem of Katz-Gabber
(\cite[Thm.~1.7.2]{Katz/LocalToGlobal}) shows that there is a
surjective morphism $\Gal(K^{\sep}/K)\twoheadrightarrow \Gal(K^{\sep}_s/K_s)$,
so we can lift $\alpha_s$ to  $\alpha\in
H^1(K,G(K^{\sep}))\cong\Hom^{\cont}(\Gal(K^{\sep}/K), \Z/p\Z)$ of order $p$, such that
$\phi(\alpha)\neq 0$.
\end{proof}

The element $\alpha$ corresponds to an Artin-Schreier extension of $K(S)$, totally
wildly ramified in $s$. Let $S'$ be the integral closure of $S$ in this Artin-Schreier extension. After
possibly shrinking $S$ around $s$, we can assume that $g:S'\rightarrow S$ is
\'etale away from $s$.

The image of $\alpha$ in $H^1(K,E_K(K^{\sep}))$ corresponds to an $E_K$-torsor
$X_K\rightarrow \Spec K$ which becomes trivial on $K(S')$. Let $f:X\rightarrow S$ be the minimal regular model of
$X_K$. In \cite[p.~82--83]{Raynaud/Driebergen} it is shown that $X$ can be constructed as a quotient of $E_{S'}$ by a finite flat equivalence relation. After perhaps shrinking $S$ around
$s$, we may assume that $f:X\rightarrow S$ is smooth away from $s$.
Let $X'\rightarrow S'$ be the normalization of $X_{S'}$. Then $X'\cong E_{S'}$. Indeed, the quotient map $E_{S'}\rightarrow X$ induces a birational, integral map  $E_{S'}\rightarrow X_{S'}$. As $E_{S'}$ is normal (even regular), this means that $E_{S'}\cong X'$.

We restrict this situation to a local setup:
Let $s'\in S'$ be a point over $s$, and write $R:=\widehat{\mathcal{O}_{S,s}}$,
$R':=\widehat{\mathcal{O}_{S',s'}}$.  We obtain the following commutative
diagram:
\begin{equation*}
	\begin{tikzcd}
		X_{R}\dar&X_{R'}\lar\dar&X'_{R'}\lar\ar{dl}\\
		\Spec R&\Spec R'\lar
	\end{tikzcd}
\end{equation*}

As $S$ is excellent,  $X'_{R'}\rightarrow X_{R'}$ is the normalization of $X_{R'}$. Indeed, we can proceed as in \cite[\href{http://stacks.math.columbia.edu/tag/07TD}{Tag 07TD}]{stacks-project}: $X'_{R'}$ is integral over $X_{R'}$, $X'_{R'}$ is normal by \cite[Prop.~6.14.1]{EGA4} , and as $X'_{R'}\rightarrow X'$ is flat, $X'_{R'}\rightarrow X_{R'}$ is birational. This means we have a factorization as in the following diagram:
\begin{equation*}
	\begin{tikzcd}		(X_{R'})^{\text{norm}}\ar[swap]{dr}{\text{normalization}}\ar[dashed]{r}{\cong}
		&X'_{R'}\dar{\text{birat.}}\rar{\text{flat}}&X'\dar{\text{normalization}}\\
		&X_{R'}\rar&X_{S'}
	\end{tikzcd}
\end{equation*}

It now suffices to prove that the finite map
\begin{equation}\label{ex:raynaud:map}
	E_{R'}\cong X'_{R'}\rightarrow
X_R\end{equation} is \'etale.

Note that $X_R$ is the minimal regular model of $X_{K_s}$
(\cite[Prop.~9.3.28]{Liu}) over $R$, and that
$X_{K_s}$ is the $E_{K_s}$-torsor associated with the cohomology class
$\alpha_s$ from above. By \cite[p.~82--83]{Raynaud/Driebergen}, the $E_{K_s}$-action on $X_{K_s}$ extends to an $E_{R}$-action on $X_R$.  
If $X_s$ and $E_s$ are the
special fibers of $X$ and $E$, then the induced action of the abelian variety $E_s$ on $X_s$
is set-theoretically transitive. This means that $(X_s)_{\red}$ is isomorphic to
$E_{s}/H$ where $H$ is some subgroup scheme of $E_s$.

Recall that we assume the existence of a finite \'etale
group scheme $G\subset E$, $G\cong (\Z/p\Z)_S$. Write $F:=E/G$. By construction
$\alpha_s$ is the image of a $G_{K_s}$-torsor, so there is a canonical
$E_{K_s}$-equivariant morphism $u_{K_s}:X_{K_s}\rightarrow F_{K_s}$. 
By functoriality of the minimal regular model, it extends to an $E_R$-equivariant
morphism  $u_R:X_R\rightarrow F_R$. We obtain an $E_s$-equivariant morphism 
\[(u_s)_{\red}:E_s/H\cong (X_s)_{\red}\rightarrow (F_s)_{\red}\cong E_s/G_s.\]
This shows that $H\subset G_s$ and that $H$ is finite \'etale over $k(s)$.
On the other hand, $G_s$ acts trivially on $(X_s)_{\red}$, as $X_R\rightarrow
\Spec R$ factors through $\Spec R'\rightarrow \Spec R$, and the residue field
extension of $R\subset R'$ is trivial. Thus $H=G_s$.

Finally, the map \eqref{ex:raynaud:map} is finite of degree $p$,
and $X'_{R'}\cong E_{R'}$. We have seen that on reduced special fibers, this map
induces an \'etale map of degree $p$:
\[E_{s'}=E_{s}\rightarrow (X_s)_{\red}=E_s/G_s.\]
This implies that \eqref{ex:raynaud:map} is \'etale.

\section{An exactness criterion for sequences of affine group schemes}\label{sec:ExCri}

In this sectio we establish an additional
technical result, which characterizes the exactness of a sequence of affine
$k$-group schemes in terms of the exactness of the induced sequences of their algebraic
quotients.
\begin{Definition}
Let $(\mathcal{T},\omega)$ be a neutral Tannakian category over a field $k$ and
let $\mathcal{S}$ be a subset of its objects. 
\begin{enumerate}
	\item 	The \emph{$\mathcal{T}$-span $\left<
				\mathcal{S}\right>_{\otimes}^{\mathcal{T}}$ of
			$\mathcal{S}$} is the smallest full sub-Tannakian
		category of $\mathcal{T}$ containing $\mathcal{S}$ which is
		closed under  subquotients and isomorphisms. When the ambient
		category is clear, we also drop the superscript $\mathcal{T}$
		and write $\left<\mathcal{S}\right>_{\otimes}$.
\item We write $\Pi_{\mathcal{S}}^{\mathcal{T}}$, or simply $\Pi_{\mathcal{S}}$
	when no confusion is possible, for the affine $k$-group scheme
	associated with $(\langle \mathcal{S}\rangle_\otimes,\omega)$ via the
	Tannaka formalism. If $\mathcal{S}=\{S\}$ is a singleton, we simply
	write $\Pi_S$ instead of $\Pi_{\{S\}}$.  
\item A sub-Tannakian category
	$\mathcal{T}'$ of $\mathcal{T}$ is called \emph{replete} if it is equal
	to the $\mathcal{T}$-span of its objects. By \cite[Prop.~2.21]{DeligneMilne} replete sub-Tannakian categories of
	$\mathcal{T}$ correspond to quotients of
	$\Pi_{\mathcal{T}}:=\aut(\omega)$.
\end{enumerate}
\end{Definition}
Recall that every affine group scheme over a field is the inverse limit of its
finite type quotients (see for example \cite[3.3, ~Cor.]{Waterhouse}). From the
Tannakian perspective, this can be seen as follows: if $(\mathcal{T},\omega)$
is a neutral Tannakian category over a field $k$, and if $S_1,S_2$ are two objects of $\mathcal{T}$, then for $i=1,2$, there are full embeddings
$\left<S_i\right>_{\otimes}\hookrightarrow \left<S_1\oplus S_2\right>_\otimes$,which induce quotient maps $\Pi_{S_1\oplus S_2}\twoheadrightarrow \Pi_{S_i}$.
These maps are the transition maps in a projective system $\{\Pi_S|S\in\mathcal{T}\}$, and \[\Pi_\mathcal{T}\cong \varprojlim_{S\in \mathcal{T}} \Pi_S.\]

\begin{Lemma}\label{lem:condSES}
	Let
\begin{equation*}
(\mathcal{T}'',\omega'')\xrightarrow{B} (\mathcal{T},\omega)\xrightarrow{A} (\mathcal{T}',\omega')
\end{equation*}
be a sequence of additive, exact, tensor functors between neutral Tannakian
categories over a field $k$, and let
\begin{equation}\label{eq:sesgp}
	\Pi_{\mathcal{T}'}\xrightarrow{a} \Pi_{\mathcal{T}} \xrightarrow{b} \Pi_{\mathcal{T}''}
\end{equation}
	be the associated sequence of affine $k$-group schemes. 
	
	The following statements are equivalent.
	\begin{enumerate}[label={\emph{(\alph*)}}, ref={(\alph*)}]
		\item\label{item:condSES:a} The sequence \eqref{eq:sesgp} is exact in the middle;
		\item\label{item:condSES:b} for every object $S\in \mathcal{T}$ the  sequence
			\begin{equation}\label{eq:sesqt}
				1\rightarrow\Pi_{A(S)}\xrightarrow{\bar{a}}\Pi_S\xrightarrow{\bar{b}}\Pi_{\mathcal{T}_S}\rightarrow 1
			\end{equation} 
of affine $k$-group schemes 
is fpqc-exact, where  $\mathcal{T}_S$ is the set of objects in $\left< S\right>_\otimes$ which are in the span of the essential image of $B$.
\end{enumerate}
\end{Lemma}
\begin{proof}
Let $\mathcal{I}$ be the essential image of $A$ and let $\mathcal{J}$ be the
essential image of $B$. Then, by \cite[Prop.~2.21]{DeligneMilne}  the sequence
\eqref{eq:sesgp} factors as
\begin{equation*}
	\begin{tikzcd}
\Pi_{\mathcal{T}'}\ar[bend left=30]{rr}{a}\rar[twoheadrightarrow]& \Pi_{\mathcal{I}}^{\mathcal{T}'}\rar[hookrightarrow] &\Pi_{\mathcal{T}} \rar[twoheadrightarrow]\ar[bend left=30]{rr}{b}&\Pi_{\mathcal{J}}^{\mathcal{T}}\rar[hookrightarrow]& \Pi_{\mathcal{T}''}.
	\end{tikzcd}
\end{equation*}
%
By replacing $\mathcal{T}'$ with
$\langle\mathcal{I}\rangle_\otimes^{\mathcal{T}'}$ and $\mathcal{T}''$ with
$\langle \mathcal{J}\rangle_\otimes^{\mathcal{T}}$, we can and will assume that $a$ is
injective (or, equivalently a closed immersion, see
\cite[15.3~Thm.]{Waterhouse}) and that $b$ is faithfully flat. 

We already remarked that
\begin{equation*}
\Pi_{\mathcal{T}}\cong\varprojlim_{S\in\mathcal{T}}\Pi_S.
\end{equation*}
Moreover, as $a$ is a closed immersion, \cite[Prop.~2.21]{DeligneMilne} shows
that for every $S'\in\mathcal{T}'$ there exists $S\in\mathcal{T}$ such that
$S'\in\langle A(S)\rangle_\otimes$, in particular the algebraic groups
$\Pi_{A(S)}$ are cofinal in the projective system
$\{\Pi_{S'}\}_{S'\in\mathcal{T}'}$ and thus
\begin{equation*}
\Pi_{\mathcal{T}'}\cong\varprojlim_{S\in\mathcal{T}}\Pi_{A(S)}.
\end{equation*}
As $b$ is faithfully flat, $B$ is  fully faithful and the essential image of $B$ is replete in $\mathcal{T}$. In particular for every object $S''\in\mathcal{T}''$ one has that 
$\Pi_{S''}=\Pi_{B(S'')}$ and thus 
\begin{equation}
\Pi_{\mathcal{T}''}\cong\varprojlim_{S\in\mathcal{T}}\Pi_{\mathcal{T}_S},
\end{equation}
where $\mathcal{T}_S$ consists of all objects in $\langle S\rangle_\otimes$ that are in the essential image of $B$. 

All together, we proved that the sequence \eqref{eq:sesgp} is the inverse limit over the
sequences \eqref{eq:sesqt}, indexed by the objects of $\mathcal{T}$.

If \eqref{eq:sesqt} is exact for every $S\in\mathcal{T}$, then \eqref{eq:sesgp}
is exact in the middle because for a projective system $\{G_i\}_i$ of $k$-group
schemes and for any commutative ring $R$, we have  $(\varprojlim
G_i)(R)=\varprojlim (G_i(R))$ and the inverse limit is a (left) exact functor.  This proves \ref{item:condSES:b} $\Rightarrow$ \ref{item:condSES:a}.

To prove \ref{item:condSES:a} $\Rightarrow$ \ref{item:condSES:b}, assume that \eqref{eq:sesgp} is exact. Recall that we may assume that $a$
is a closed immersion and that $b$ is faithfully flat. We get the following commutative diagram
\begin{equation}\label{eq:condSES:diag}
	\begin{tikzcd}
1\ar{r}&
\Pi_{\mathcal{T}'}
\ar{r}{a}
\ar{d}{\pr'}&
\Pi_{\mathcal{T}}
\ar{r}{b}
\ar{d}{\pr}&
\Pi_{\mathcal{T}''}
\ar{r}
\ar{d}{\pr''}&
1\\ 1
\ar{r}&\Pi_{A(S)}\ar{r}{\bar{a}}&\Pi_S\ar{r}{\bar{b}}&
\Pi_{\mathcal{T}_S}\ar{r}& 1
	\end{tikzcd}
\end{equation}
where all vertical arrows are faithfully flat. We want to prove that the bottom
row is exact.  It follows directly from \cite[Prop.~2.21]{DeligneMilne} that the morphism $\bar{a}$ is a
closed immersion and that $\bar{b}$ is faithfully flat.  It is also clear that $\bar{b}\cdot\bar{a}$
is the trivial morphism. It remains to prove that \eqref{eq:sesqt} is exact in
the middle. 

We can write $\Pi_S$ as $\Pi_{\mathcal{T}}/H$ and $\Pi_{A(S)}$ as
$\Pi_{\mathcal{T}'}/H'$, for some normal subgroup schemes $H,H'$.  As $a$ is injective, we see that $\ker(\pr\circ a)=\Pi_{\mathcal{T}'}\cap H$, and as $\bar{a}$ is injective,
$H'=\ker(\bar{a}\circ\pr')$. As \eqref{eq:condSES:diag} commutes, the two kernels must agree and thus $H'=\Pi_{\mathcal{T}'}\cap H$.

In particular, this shows that $\Pi_{A(S)}$ is normal in $\Pi_S$. Thus to prove
that \eqref{eq:sesqt} is exact, it is enough to prove that
$\Pi_{\mathcal{T}_S}$ is the cokernel of $\bar{a}$. By the universal property
of the cokernel, together with \cite[Cor.~2.9, Prop.~2.21]{DeligneMilne}, the
cokernel of $\bar{a}$ corresponds uniquely to a replete sub-Tannakian category
$\mathcal{C}$ of $\langle S\rangle_\otimes$, characterized by the property that
for every $S'\in\langle S\rangle_\otimes$, $A(S')$ is trivial if and only if $S'\in
\mathcal{C}$, where by \emph{trivial} we mean that $A(S')$ is the direct sum of
copies of the unit object of $\mathcal{T}''$.  Our assumptions imply that
$\Pi_{\mathcal{T}''}$ is the cokernel of $a$, hence if $S'\in\mathcal{T}$,
$A(S')$ is trivial if and only if $S'\in \mathcal{T}''$, hence by definition
$\mathcal{T}_S$ is the maximal replete sub-Tannakian category of $\langle
S\rangle_\otimes$ whose objects are trivialized by $A$, which finishes the
proof.

\end{proof}

\section{The homotopy exact sequence for regular singular stratified bundles}\label{sec:Homty}

We return to the notations that were introduced in \Cref{sec:Rs-strat}. In particular, $k$ is an algebraically closed field of characteristic $p>0$ and  $X$ is a smooth, connected,  separated scheme of finite type over  $k$.
We recall that given $x\in X(k)$, we have associated $k$-group schemes $\Pi(X,x)$, $\Pi^{\rs}(X,x)$ and $\Pi^{\rs}((X,\overline{X}),x)$, where $(X,\overline{X})$ is a good partial
compactification of $X$ (\Cref{defn:stratifiedBundles}).

Note that according to \cite[Prop.~4.5]{Kindler/FiniteBundles},
$\Strat^{\rs}((X,\overline{X}))$ is a replete sub-Tannakian category of
$\Strat(X)$ and hence there is a  quotient map $\Pi(X,x)\twoheadrightarrow
\Pi^{\rs}( (X,\overline{X}),x)$.

\begin{Theorem}\label{thm:sesFunG}
Let $X$ and $Y$ be smooth, connected, separated $k$-schemes of finite type and let $f:Y\to X$ be a smooth
projective morphism with geometrically connected fibers. Fix good partial
compactifications $(X,\overline{X})$ and $(Y,\overline{Y})$ and write $D_X:=(\overline{X}\setminus X)_{\red}$, $D_Y:=(\overline{Y}\setminus Y)_{\red}$.
Let $\bar{f}:\overline{Y}\to\overline{X}$ be an extension of $f$ satisfying:
\begin{enumerate}[label={\emph{(\alph*)}}, ref={(\alph*)}]
	\item\ $\bar{f}(\overline{Y})$ contains all generic points of $D_X$.
	\item $\bar{f}$ is log smooth when $\overline{Y}$ and $\overline{X}$ are
		equipped with the log structures defined by $D_Y$, resp.~$D_X$ (\Cref{rem:log-smoothness}).				
\end{enumerate}
If $y\in Y$ is a closed point, define $x:=f(y)$ and let $j:Y_x\hookrightarrow \overline{Y}$ be the fiber over $x$. Then the sequence 
\begin{equation}\label{eq:sesFunGp}
\Pi(Y_x,y)\xrightarrow{j^*}\Pi^{\rs}( (Y,\overline{Y}),y)\xrightarrow{f^*} \Pi^{\rs}( (X,\overline{X}),x)\to 1
\end{equation}
of affine $k$-group schemes is fpqc exact.
\end{Theorem}
\begin{proof}
	Notice that the assumptions of the theorem imply that $Y_x$ is proper, hence all stratified bundles  on  $Y_x$ are trivially regular singular.
	We will deduce the theorem from \cite[Thm.~1]{DosSantos/HomotopySES},
	which states that under our hypotheses the sequence
	\[\Pi(Y_x, y)\rightarrow \Pi(Y,y)\rightarrow \Pi(X,x)\rightarrow 1\]
	is exact.
	In particular, the top horizontal arrow in the  commutative diagram 
	\begin{equation*}
\begin{tikzcd}
\Pi(Y,y)\ar{r}{f^*}\ar{d}{\pi_Y}&\Pi(X,x)\ar{d}{\pi_X}\\
\Pi^{\rs}( (Y,\overline{Y}),y)\ar{r}{f^*} &\Pi^{\rs}( (X,\overline{X}),x)
\end{tikzcd}
\end{equation*}
is faithfully flat, and so are the vertical arrows. This shows that the bottom
horizontal arrow is faithfully flat as well.

If $E\in \Strat^{\rs}( (Y,\overline{Y}))$, then \cite[Prop.~4.5]{Kindler/FiniteBundles} shows that the span $\left<E\right>_{\otimes}$ is independent of whether we compute it in $\Strat^{\rs}((Y,\overline{Y}))$ or the ambient category $\Strat(Y)$.

Similarly, if $E\in\Strat^{\rs}( (Y,\overline{Y}))$, then we can consider the span
of $j^*E$ in $\Strat(Y_x)$  and we write $\Pi_{j^*E}=\Pi_{j^*E}^{\Strat(Y_x)}$
for the affine $k$-group scheme attached to it. By \Cref{lem:condSES}, to prove
that \eqref{eq:sesFunGp} is exact, it is
enough to prove that for every $(Y,\overline{Y})$-regular singular stratified
bundle $E$, the sequence
\begin{equation}\label{eq:sesMonGpRS}
1\to\Pi_{j^*E}\to\Pi_{E}\to\Pi_{\mathcal{T}^{\rs}_E}\to 1
\end{equation}
is exact, where $\mathcal{T}^{\rs}_E$ are the objects of $\langle
E\rangle_\otimes$ that are also in $\Strat^{\rs}(X,\overline{X})$ (seen as a full replete
sub-Tannakian category of $\Strat^{\rs}(Y,\overline{Y})$ via $f^*$).  

By \cite[Thm.~1]{DosSantos/HomotopySES} the sequence
 \begin{equation*}
 \Pi(Y_x,y)\xrightarrow{j^*}\Pi(Y,y)\xrightarrow{f^*} \Pi(X,x)\rightarrow 1
\end{equation*}
is exact. Hence, the sequence
\begin{equation}\label{eq:sesMonGp}
1\rightarrow\Pi_{j^*E}\to\Pi_E\to\Pi_{\mathcal{T}_E}\rightarrow 1
\end{equation}
is exact for every $E\in \Strat^{\rs}((Y,\overline{Y}))$ (\Cref{lem:condSES}), where $\mathcal{T}_E$ are the objects of
$\langle E\rangle_\otimes$ that are also in $\Strat(Y)$. By
\Cref{cor:checkRSafterPullback}, we have $\mathcal{T}_E=\mathcal{T}^{\rs}_E$, so the
sequence \eqref{eq:sesMonGp} is isomorphic to the sequence
\eqref{eq:sesMonGpRS}, which is then exact.
Using \Cref{lem:condSES} again, this concludes the proof.
\end{proof}

\begin{Remark}\label{rmk:counterex}
The proof above shows a little more: namely, if there exists  a stratified
bundle $E$ on $X$ which is not $(X,\overline{X})$-regular singular but such that $f^*E$
is $(Y,\overline{Y})$-regular singular (e.g.~as in the example of Section~\ref{sec:ExRa}), then
the sequence \eqref{eq:sesFunGp} cannot be exact.

To see this, note that under the conditions of \Cref{thm:sesFunG} the surjectivity of 
\begin{equation*}
f^*:\Pi^{\rs}(Y,\overline{Y},y)\to\Pi^{\rs}(X,\overline{X},x)
\end{equation*}
holds even without the log smoothness hypothesis on $\bar{f}$. This allows us
to see $\Strat(X)$ and $\Strat^{\rs}(X,\overline{X})$  as replete sub-Tannakian
categories of $\Strat(Y)$.
As in the proof of \Cref{thm:sesFunG}, consider the full subcategories
\[\mathcal{T}_{f^*E}:=\left<f^*E\right>_{\otimes}\cap \Strat(X)\subset \Strat(Y),\] and 
\[\mathcal{T}^{\rs}_{f^*E}:=\left< f^*E\right>_\otimes\cap \Strat^{\rs}( (X,\overline{X}))\subset \Strat(Y).\] 

The exactness of \eqref{eq:sesMonGp} for $f^*E$ reduces to the isomorphism
$\Pi_{f^*E}\cong \Pi_{\mathcal{T}^{\rs}_{f^*E}}$. 
Note that indeed by \cite[Thm.1]{DosSantos/HomotopySES} and  \Cref{thm:sesFunG} one has that $\Pi_{f^*E}\cong\Pi_{\mathcal{T}_{f^*E}}$. On the other hand, if $E$ is not $(X,\overline{X})$-regular singular, then $f^*E\not\in
\mathcal{T}^{\rs}_{f^*E}$, so the inclusion $\mathcal{T}^{\rs}_{f^*E}\subset \mathcal{T}_{f^*E}$ is strict. 
This means that the induced morphism $\Pi_{f^*E}\rightarrow
\Pi_{\mathcal{T}^{\rs}_{f^*E}}$ is not an isomorphism, so  it follows from \Cref{lem:condSES} the sequence \eqref{eq:sesFunGp} cannot be exact in the middle. This shows that the log smoothness assumption in \Cref{thm:sesFunG} cannot be dropped entirely.
\end{Remark}

A stratified bundle $E$ on $X$  is called \emph{regular singular} if it is
$(X,\overline{X})$-regular singular for every good partial compactification
$\overline{X}$. We denote by $\Pi^{\rs}(X,x)$ the Tannakian fundamental group
associated with the full subcategory of $\Strat(X)$ given by all
regular singular stratified bundles (\cite[Prop.~7.4]{Kindler/FiniteBundles}). By \cite[Prop.~7.5]{Kindler/FiniteBundles}
if $X$ admits a good compactification $\overline{X}$, then a stratified bundle
$E$ is regular singular if and only if it is $(X,\overline{X})$-regular
singular. Together with \cite[Thm.~1.3]{Kindler/Pullback} this implies the
following.

\begin{Corollary}
Retain the notations and assumptions of \Cref{thm:sesFunG}, and assume furthermore that $\overline{X}$ is proper. Then the sequence
\begin{equation*}
\Pi(Y_x,y)\xrightarrow{j^*}\Pi^{\rs}(Y,y)\xrightarrow{f^*} \Pi^{\rs}(X,x)\to 1
\end{equation*}
is fpqc exact.
\end{Corollary}

\begin{Corollary}[{K\"unneth formula}]
Let $X$ and $Y$ be smooth, connected $k$-varieties with $Y$ projective, let $x\in X(k)$, $y\in Y(k)$ and let $z\in (Y\times_k X)(k)$ be the point induced by $x$ and $y$. Then, the natural morphism induced by the projections
\[\Pi^{\rs}(X\times_k Y,z)\to\Pi^{\rs}(X,x)\times \Pi^{\rs}	(Y,y)\]
is an isomorphism.
\end{Corollary}
\begin{proof}Note that $Y$ is proper, hence $\Pi(Y,y)=\Pi^{\rs}(Y,y)$.

The strategy of the proof is exactly the same as in \cite[X, Cor.~1.7]{SGA1}.
In order to make use of \Cref{thm:sesFunG}, we only need to remark that for every partial good compactification $\overline{X}$ of $X$ there exists a good partial compactification of $X\times Y$ (log) smooth over $\overline{X}$, namely $\overline{X}\times Y$. 
\end{proof}

\bibliographystyle{amsalphacustomlabels}
\bibliography{alggeo}

\end{document}